\documentclass[10pt]{amsart}
\usepackage{amssymb,amsbsy,amsmath,amsfonts,times}
\usepackage{latexsym,euscript,exscale}

\usepackage{helvet}

\title{Completely metrisable groups acting on trees}
\author {Christian Rosendal}
\address{Department of Mathematics, Statistics, and Computer Science (M/C 249)\\
University of Illinois at Chicago\\
851 S. Morgan St.\\
Chicago, IL 60607-7045\\
USA}
\email{rosendal@math.uic.edu}
\urladdr{http://www.math.uic.edu/$~$rosendal}

\date {}
\newcommand {\N}{\mathbb N}
\newcommand {\Q}{\mathbb Q}
\newcommand {\R}{\mathbb R}
\newcommand {\Z}{\mathbb Z}

\newcommand {\U}{\mathbb U}

\newcommand {\D}{\mathbb D}

\newcommand{\iso}{\cong}

\newcommand{\tom} {\emptyset}

\newcommand{\inj}{\hookrightarrow}

\newcommand{\saa}{\Rightarrow}
\newcommand{\equi}{\Longleftrightarrow}

\newcommand{\til}{\rightarrow}

\newcommand {\del}{ \; \big| \;}

\newcommand {\go} {\mathfrak}

\newcommand{\inv}{^{-1}}

\newcommand {\e} {\exists}
\renewcommand {\a} {\forall}

\newtheorem{thm}{Theorem}[section]
\newtheorem{cor}[thm]{Corollary}
\newtheorem{lemme}[thm]{Lemma}
\newtheorem{prop} [thm] {Proposition}
\newtheorem{defi} [thm] {Definition}

\newtheorem{claim}[thm] {Claim}

\newtheorem{conj}[thm]{Conjecture}

\usepackage{fouriernc}

\begin{document}
\thanks{The research of the author was partially supported by NSF grants DMS 0901405 and DMS 0919700}

\begin{abstract}We consider actions of completely metrisable  groups on simplicial trees in the context of the Bass--Serre theory. Our main result characterises continuity of the amplitude function corresponding to a given action. Under fairly mild conditions on a completely metrisable  group $G$, namely, that the set of elements generating a non-discrete or finite subgroup is somewhere dense, we show that in any decomposition as a free product with amalgamation, $G=A\ast_CB$, the amalgamated groups $A$, $B$ and $C$ are open in $G$.\end{abstract}

\maketitle
\tableofcontents

\section{Introduction} 
While Klein in his Erlangen programme insisted that geometric structures should be understood through their symmetry groups, one of the main tenets of geometric group theory is that abstract groups should be understood by via their actions on geometric structures. A particular instance of this is the Bass--Serre theory of actions on trees \cite{serre}, which gives a complete understanding of free products with amalgamation and HNN extensions of groups. The present paper specialises the Bass--Serre theory to common topological groups and can be seen as an attempt to provide a coherent approach to various phenomena from the literature indicating a certain rigidity of completely metrisable  groups with respect to free constructions and actions on trees. We are specifically aiming to gain a better understanding of the following result.
\begin{thm}[R. M. Dudley \cite{dudley}]\label{dudley}
Let $G$ be a completely metrisable or locally compact, Hausdorff, topological group. Assume that 
$$
\pi \colon G\til F
$$
is a homomorphism into a free group (free Abelian, or free non-Abelian). Then $\pi$ is continuous with respect to the discrete topology on $F$, i.e., 
${\rm ker}(\pi)$ is an open subgroup of $G$.
\end{thm} 
As an immediate application, one sees that the only completely metrisable or locally compact, Hausdorff, group topology on a free group is the discrete one. Going beyond Dudley's result, it is natural to attempt to weaken the conditions on $F$. Now, by the Nielsen--Schreier--Serre Theorem, a group is free if and only if it acts freely and without inversion on a tree, so we might replace the homomorphism $\pi$ with an action of $G$ without inversion on a tree. A number of results in the literature treat various aspects of this problem, see, e.g., \cite{alperin1,alperin2,bass,guirardel,morris,ivanov1,kossak,tits,mactho,shelah}, notably the results of R. Alperin \cite{alperin2}  entirely elucidating the situation for locally compact, Hausdorff groups (for ease of exposition, henceforth, all locally compact groups will be assumed to be Hausdorff).

Our main result isolates  a very natural property of actions of completely metrisable  groups on trees, which we will indicate might be inevitable. Before stating it, let us stress that, though the groups we consider are topological groups, we never make any assumptions as to the continuity of their actions. Determining to which extent they are necessarily continuous is, in fact, one of the main objectives of this paper.

\begin{thm}\label{equiv}Let $G$ be a completely metrisable or locally compact group acting without inversion on  a tree $X$. Then the following four properties are equivalent.
\begin{enumerate}
\item $G$ either fixes an end or has an open subgroup fixing a vertex,
\item the amplitude $\|\cdot\|\colon G\til \N$ is a continuous function on $G$,
\item the set of elliptic isometries is somewhere comeagre in $G$,
\item there is an open subgroup consisting of elliptic isometries.
\end{enumerate}
Moreover, the above properties hold for any action of $G$ if and only if whenever $G$
is written as a non-trivial free product with amalgamation, $G=A*_CB$,  the subgroups $A$, $B$, and $C$ are open in $G$.
\end{thm}
We recall that a subset of $G$ is {\em comeagre} in an open set $U\subseteq G$ if it contains the intersection of countably many dense open subsets of $U$. Also, a subset of $G$ is {\em somewhere comeagre} if it is comeagre in a non-empty open set. A free product with amalgamation, $G=A*_CB$, is {\em non-trivial} if $A$ and $B$ are proper subgroups of $G$. Finally, an isometry of a tree is {\em elliptic} if it fixes a vertex.

To simplify exposition, we shall refer to the equivalent properties of Theorem \ref{equiv} under one name.
\begin{defi}
A topological group $G$ is said to have {\em property (OA)} (for open amalgams), if whenever $G=A*_CB$ is a non-trivial decomposition, the three subgroups $A$, $B$ and $C$ are open. 
\end{defi}

With this definition, we can formulate one of the main results of  \cite{alperin2} as follows.
\begin{thm}[R. Alperin \cite{alperin2}]
Locally compact groups have property (OA).
\end{thm}

Using Theorem \ref{equiv}, we show that the same holds for a fairly substantial class of completely metrisable groups.

\begin{thm}\label{somewhere comeagre}
Let $G$ be a completely metrisable group such that the set
$$
\D=\{g\in G\del g\text{ generates a finite or non-discrete subgroup of }G\}
$$
is somewhere dense in $G$. Then $G$ has property (OA). 
\end{thm}

Though our main interest is in Polish, i.e., separable, completely metrisable groups, none of our results rely on separability and we have therefore chosen to formulate them in the more general setting.

Based on the results of Alperin and Dudley mentioned above, along with Theorems \ref{equiv} and \ref{somewhere comeagre}, we are led to the following conjecture, which we shall also provide additional evidence for.
\begin{conj}\label{conj}
Any completely metrisable topological group has property (OA).
\end{conj}

The paper is organised as follows: In Section \ref{trees and actions} we collect some of the background material on trees and groups acting on them that we will need in the paper. Nothing here is novel and this section can easily be skipped by the reader familiar with Bass--Serre theory. In Section \ref{main}, we prove Theorem \ref{equiv} and finally in Section \ref{non-discrete} we give the proof of Theorem \ref{somewhere comeagre} along with some examples of groups satisfying its premises.

\

\noindent{\bf Acknowledgement:} I wish to thank the anonymous referee for a careful reading and many helpful suggestions.

\section{Actions on trees}\label{trees and actions}
In this section we will recall basic Bass--Serre theory for which our basic reference is Serre's book \cite{serre}. None of the results here are novel and can almost all be found explicitly in \cite{serre}.
\subsection{Graphs and trees}
A graph $X$ consists of a set of vertices, ${\rm Vert}\;X$, and a set of edges, ${\rm Edge}\;X$, along with three functions $o,t\colon {\rm Edge}\; X\til {\rm Vert}\; X$, and $\overline\cdot\colon {\rm Edge}\; \til {\rm Edge}\; X$ such that
$$
o(e)=t(\overline e),
$$
$$
e\neq \overline e,
$$
and 
$$
\overline{\overline e}=e.
$$  
So each edge $e$ comes as a pair, $e$ and its {\em inverse} $\overline e$, whence $e$ is a directed vertex from its {\em origin} $o(e)$ to its {\em terminal} node $t(e)$. 
An {\em edge path} in $X$ is a finite, infinite or biinifinite sequence $(e_n)$ of edges such that $t(e_n)=o(e_{n+1})$ for all $n$. The path is {\em reduced} if for all $n$, $\overline e_n\neq e_{n+1}$. A finite path $(e_1,e_2,\ldots,e_n)$ is a {\em loop} if $o(e_1)=t(e_n)$. 

A {\em tree} $X$ is  a connected graph without non-trivial, reduced loops. Since a tree $X$ is connected and acyclic, it is uniquely path connected, i.e., for any distinct $x,y\in {\rm Vert}\;X$ there is a unique reduced path $(e_1,\ldots,e_n)$ with $o(e_1)=x$ and $t(e_n)=y$. This path, denoted by $x-y$, is called the {\em geodesic} from $x$ to $y$ and lets us define metric on ${\rm Vert}\;X$ by setting
$$
d(x,y)=n.
$$

A {\em vertex path} or {\em line segment} is a finite, infinite or biinfinite sequence $(x_n)$ of vertices of $X$ such that for every $n$ there in an edge $e\in {\rm Edge}\;X$ with $o(e)=x_n$ and $t(e)=x_{n+1}$. A {\em line} in $X$ is a biinfinite vertex path
$$
\ell=(\ldots,x_{-2},x_{-1},x_0,x_1,x_2,\ldots)
$$
of distinct vertices of $X$, and, similarly, a {\em half-line} is an infinite path $\ell^+=(x_0,x_1,x_2,\ldots)$ of distinct vertices. When convenient, we shall identify the geodesic $x-y=(e_1,\ldots,e_n)$ with the corresponding vertex path $(x,t(e_1),t(e_2),\ldots,t(e_n))$. 

An {\em (induced) subgraph} of a graph $X$ is a graph $Y$ such that ${\rm Vert}\; Y\subseteq {\rm Vert}\;X$ and ${\rm Edge}\;Y=\{e\in {\rm Edge}\:X\del o(e),t(e)\in{\rm Vert}\;Y\}$. A subgraph $Y\subseteq X$ of a tree $X$ is said to be a {\em subtree} if, moreover, for any $x,y\in {\rm Vert}\;Y$ the geodesic path $x-y$ is contained in $Y$, i.e., if $Y$ is itself connected. For simplicity of notation, we shall sometimes identify a subgraph with its set of vertices, noticing that its edges can then be read off from the full graph.
It follows from the unique path connectedness of a tree $X$ that if $X_1,\ldots,X_n$ is a finite family of subtrees of $X$ that intersect pairwise, we have
$$
X_1\cap X_2\cap\ldots\cap X_n\neq \tom
$$
(see Lemma 10, p. 65 \cite{serre}). Actually, we shall also need the following more general fact due to J. Tits (see Lemma 1.6 \cite{tits2}). Whenever $(X_i)_{i\in I}$ is a family of subtrees that intersect pairwise, then either $\bigcap_{i\in I}X_i\neq \tom$ or there is an infinite half-line $\ell^+\subseteq X$ that is eventually contained in each $X_i$, i.e., such that for every $i\in I$, $\ell^+\setminus X_i$ is finite.

\subsection{Single automorphisms}
An isometry, or equivalently an automorphism, $g$ of a tree $X$ is said to be {\em without inversion} if there is no edge $e\in {\rm Edge}\:X$ with $g(e)=\overline e$. In this case, there are two possibilities for $g$; either $g$ fixes a vertex of $X$, in which case $g$ is said to be {\em elliptic}, or there is a unique biinfinite line $\ell_g$ on which $g$ acts by non-trivial translation, in which case $g$ is said to be {\em hyperbolic}. For isometries $g$ without inversion, we associate the {\em characteristic subtree} $X_g$ of $X$, which, if $g$ is elliptic, is the non-empty set of vertices fixed by $g$ and, if $g$ is hyperbolic, is its axis of translation, $\ell_g$.

We define the {\em amplitude} of $g$ by the formula
$$
\|g\|=\min\big(d(g(y),y)\del y\in {\rm Vert}\;X\big),
$$
so $g$ is elliptic if and only if $\|g\|=0$.
Now, if $g$ is a hyperbolic isometry of $X$ with characteristic subtree $\ell_g$, then we can write $\ell_g=(\ldots,x_{-2},x_{-1},x_0,x_1,x_2,\ldots)$ such that for some $n\geqslant 1$ and all $m\in \Z$
$$
g(x_m)=x_{m+n}.
$$
In this case, we have $n=\|g\|$ and for any $y\notin \ell_g$, if $x_m$ is the vertex of $\ell_g$ closest to $y$, the geodesic path $y-g(y)$ from $y$ to $g(y)$, is the concatenation of the geodesic paths $y-x_m$, $x_m-x_{m+n}$ and $x_{m+n}-g(y)$.

Suppose $g$ and $h$ are isometries of $X$ without inversion. Then the following equality can easily be checked by inspection  (see, e.g., M. Culler and J. W.  Morgan \cite{culler}).
\begin{quote}
If $X_g\cap X_h=\tom$, then $\|gh\|=\|g\|+\|h\|+2\,{\rm dist}(X_g,X_h)$.
\end{quote}
From this, one immediately obtains what is sometimes referred to as {\em Serre's Lemma} (see Corollary 2, p. 64 \cite{serre}).
\begin{quote}
If $g$, $h$ and $gh$ are all elliptic, then $X_g\cap X_h\neq \tom$.
\end{quote}

We also define the following equivalence relation on the half-lines of $X$. If $\ell^+=(x_0,x_1,x_2,\ldots)$ and $l^+=(y_0,y_1,y_2,\ldots)$ set
$$
\ell^+\sim l^+\equi \e n ,m\; (x_n=y_m\;\&\; x_{n+1}=y_{m+1}\;\&\;\ldots).
$$
So $\ell^+$ and $l^+$ are equivalent if they share a common tail. The equivalence classes of half-lines are called {\em ends} of $X$ and we see that any isometry of $X$ naturally defines a permutation of the set of ends by $g\cdot [\ell^+]=[g\cdot \ell^+]$.

Note that if ${\bf e}=[\ell^+]$, for $\ell^+=(x_0,x_1,\ldots)$, is an end fixed by an isometry $g$ of $X$, we have $\ell^+\sim g\cdot \ell^+$, and so one of the following three things will happen;
\begin{enumerate}
  \item for some $n$ and all $m\geqslant n$, $g(x_m)=x_m$,
  \item for some $n$, $g(\{x_n,x_{n+1},\ldots\})\subseteq \{x_{n+1},x_{n+2},\ldots\}$,
  \item for some $n$, $\{x_n,x_{n+1},\ldots\}\subseteq g(\{x_{n+1},x_{n+2},\ldots\})$.
\end{enumerate}
In the first case, $g$ is elliptic and we say that $\bf e$ is a {\em neutral fixed end} for $g$. In the two other cases, $g$ will be hyperbolic with its axis containing a tail of $\ell^+$. In the second case, we say that $\bf e$ is an {\em attracting fixed end} for $g$, and in the last case, $\bf e$ is a {\em repulsing fixed end} for $g$.

\subsection{Group actions}
If a group $G$ acts without inversion on a tree $X$, we define the quotient graph $G\setminus X$ by
\begin{align*}
&{\rm Vert}\; G\setminus X= \{Gx\del x\in {\rm Vert}\; X\},\\
&{\rm Edge}\; G\setminus X=\{Ge\del e\in {\rm Edge}\; X\},\\
&o(Ge)=Go(e),\\
&t(Ge)=Gt(e), \\
&\overline{Ge}=G\overline e.
\end{align*}

Note that from Serre's Lemma it follows that if $G$ is a group of elliptic isometries, then $X_g\cap X_h\neq \tom$ for all $g,h\in G$. So either $\bigcap_{g\in G}X_g\neq \tom$, in which case $G$ fixes a vertex of $X$, or $\bigcap_{g\in G}X_g=\tom$, in which case it can be checked that $G$ has a neutral fixed end ${\bf e}=[\ell^+]$, $\ell^+=(x_0,x_1,\ldots)$ (see Proposition 3.4 \cite{tits1}). In the second case, letting $G_n=\{g\in G\del \a m\geqslant n\; g(x_m)=x_m\}$, we see that $(G_n)_{n\in \N}$ is an increasing chain of proper subgroups of $G$ whose union is $G$.
Conversely, if $G$ can be written as the union of an increasing chain $(G_n)$ of proper subgroups, we can define a tree $X$ with vertex set 
$$
{\rm Vert}\:X=G/G_0\sqcup G/G_1\sqcup G/G_2\sqcup\ldots
$$
and with edges between the vertices $gG_n$ and $gG_{n+1}$ for $g\in G$ and $n\in \N$. Then $G$ acts without inversion on $X$ by left translation of the cosets $gG_n$ and, since $G=\bigcup_{n\in\N}G_n$, it has a neutral fixed end ${\bf e}=[\ell^+]$, namely $\ell^+=(G_0, G_1,G_2,\ldots)$, but does not  fix a vertex.

This shows that for any group $G$ the following two conditions are equivalent
\begin{itemize}
\item $G$ acts without inversion on a tree with a neutral fixed end, but not fixing a vertex,
\item $G$ is the union of a countable increasing chain of proper subgroups.
\end{itemize}

Now suppose instead  $G$ can be written as a free product with amalgamation $G=A*_CB$. Then we can define a tree $X$ (see Theorem 7, p. 32 \cite{serre}) whose vertex set is the space of left cosets $G/A\sqcup G/B$ and connecting any pair $gA$ and $gB$ by a set of edges. In other words, $gA$ and $hB$ are connected by an edge if and only if $gA\cap hB\neq \tom$. Again, $G$ acts without inversion on $X$ by left-translation of the cosets.
Note also that the stabiliser in $G$ of a vertex $hA$ is just the subgroup $hAh\inv\leqslant G$, while the stabiliser of $hB$ is the subgroup $hBh\inv\leqslant G$. It follows that the stabiliser of the edge between $hA$ and $hB$ is
$$
hAh\inv \cap hBh\inv=h(A\cap B)h\inv=hCh\inv.
$$
Suppose, moreover, that the product is non-trivial and pick $a\in A\setminus B$ and $b\in B\setminus A$.
It follows that the product $ab$ does not belong to any conjugate of either $A$ or $B$ and, in particular, $ab$ does not fix any vertex of $X$. So $ab\in G$ is hyperbolic.

In the course of the proof of Theorem \ref{equiv}, we shall also be needing the following result of Bass-Serre theory (Corollary 1, p. 55 \cite{serre}).
\begin{thm}\label{serre2}
Suppose $G$ is a group acting without inversion on a tree $X$ and let $R\leqslant G$ be the subgroup of $G$ generated by the elliptic isometries. Then $\pi_1(G\setminus X)\iso G/R$. In particular, since $\pi_1(G\setminus X)$ is a free group, either 
\begin{itemize}
\item $G\setminus X$ is a tree, in which case $\pi_1(G\setminus X)=\{1\}$ and $G=R$, or
\item $\pi_1(G\setminus X)\iso G/R$ is a free group of rank $\geqslant 1$, whence $G$ has an infinite cyclic quotient.
\end{itemize}
\end{thm}


\section{Actions of completely metrisable and locally compact groups}
\subsection{Discontinuous and definable actions}\label{examples}
Before we begin the proof of Theorem \ref{equiv}, we shall first provide some examples providing empirical evidence for Conjecture \ref{conj}.
At the same time, these examples delimit the type of continuity that one can hope for in general.

\begin{itemize}
\item A completely metrisable, locally compact group acting without inversion with no open vertex stabiliser.
\end{itemize}
To construct such an example, let $G$ be the union of a countable chain of non-open proper subgroups 
$$
G_0<G_1<G_2<\ldots<G=\bigcup_{n\in \N}G_n.
$$ 
Then when $G$ acts on the corresponding tree $X$ with vertex set 
$$
{\rm Vert}\:X=G/G_0\sqcup G/G_1\sqcup G/G_2\sqcup\ldots
$$
the vertex stabilisers are simply conjugates of the groups $G_n$, none of which are open. For example, if $B$ is a Hamel basis for $\R$ as a $\Q$-vector space, write $B$ as the union of a strictly increasing chain of proper subsets $B_n$. Then if $G_n$ is the $\Q$-vector subspace generated by $B_n$, the $G_n$ form an increasing, exhaustive chain of proper non-open subgroups of $G$. However, in this case, there is a neutral fixed end for $G$.

\begin{itemize}
\item A completely metrisable, locally compact  group acting without inversion, neither fixing an end nor having open vertex stabilisers.
\end{itemize}
For this example, let $G$ be non-discrete  and let $X$ be the tree with vertices $\{*\}\cup G$, where $*$ is a point not in $G$, and equip $X$ with edges between $*$ and $g$ for all $g\in G$. Then $G$ acts without inversion on $X$ by fixing $*$ and otherwise acting by left translation on the vertices $g\in G$.
Thus, since $G$ is not discrete, the stabilisers ${\rm stab}_G(g)=\{1\}$ fail to be open. Moreover, there is no end to fix. Nevertheless, in this example, $G$ has an open subgroup fixing a vertex of $X$, namely $*$. In the concrete example, $G$ is itself the open subgroup in question, but this can also be avoided with more care.

\begin{itemize}
\item $\Z^\N$ has no open subgroup with property (FA').
\end{itemize}
We recall that a group has property (FA') if whenever it acts without inversion on a tree, every element is elliptic. Since any open subgroup of $\Z^\N$ admits an epimorphism onto $\Z$, no open subgroup satisfies property (FA'). As we shall see later, this distinguishes completely metrisable groups from locally compact groups.

\

For completeness, we show that in the definable setting, some of the above discontinuities disappear (see \cite{kechris} for the basic concepts and results of descriptive set theory).
\begin{prop}
Suppose $X$ is a standard Borel tree, i.e., ${\rm Vert}\; X$ and ${\rm Edge}\:X$ are standard Borel spaces and the functions $o,t, \overline\cdot$ are Borel.
Let $G$ be a Polish, i.e., separable, completely metrisable group acting without inversion on $X$ such that the action  $G\curvearrowright {\rm Vert}\;X$ is Borel. Then there is an open subgroup $H\leqslant G$ fixing a vertex.
\end{prop}

\begin{proof}Fix some vertex $x_0\in {\rm Vert}\;X$. 
We note that for any $n$,
\begin{align*}
B_n=&\{g\in G\del d(gx_0,x_0)\leqslant n\}\\
=&\bigcup_{m\leqslant n}\{g\in G\del \e e_1,e_2,\ldots,e_m\in {\rm Edge}\:X\; t(e_i)=o(e_{i+1})\;\&\; x_0=o(e_1)\;\&\; gx_0=t(e_m)    \}
\end{align*}
is analytic. Since $\bigcup_nB_n=G$, some $B_n$ is non-meagre, whence by Pettis' Theorem (see Theorem (9.9) in \cite{kechris}) $B_nB_n\inv=B_nB_n\subseteq B_{2n}$ is a neighbourhood of the identity in $G$. Now pick some neighbourhood $U$ of the identity in $G$ such that $U^{2n+1}\subseteq B_{2n}$. Then, every element of $U$ is elliptic. For otherwise, if $g\in U$ is hyperbolic of amplitude $\|g\|>0$, then $g^{2n+1}\in U^{2n+1}$ is hyperbolic of amplitude $>2n$, whence $d(gx_0,x_0)>2n$, contradicting $g\in B_{2n}$. Now pick a neighbourhood $W\subseteq U$ of the identity such that $W^2\subseteq U$. Then every element of $W$ is elliptic and, moreover, if $g,f\in W$, then also $gf\in U$ is elliptic. By Serre's Lemma, $X_g\cap X_f\neq \tom$ for any $g,f\in W$. Also,  if $g\in W$, then, since $d(gx_0,x_0)\leqslant 2n$, we have $d(x_0,X_g)\leqslant n$. It follows that $\bigcap_{g\in W}X_g\neq \tom$, and so, if $H$ is the open subgroup of $G$ generated by $W$, $H$ fixes any vertex in $\bigcap_{g\in W}X_g$.
\end{proof}


\subsection{Actions on a line}
Dudley's Theorem \ref{dudley} in particular implies that if a completely metrisable or locally compact group $G$ acts by translations on a biinfinite line $\ell$, then the kernel of the action is open in $G$. Now, if we also allow reflections, this no longer holds, since $G$ could have a non-open subgroup $H\leqslant G$ of index $2$, whence $G/H\iso \Z_2$ and hence also $G$ act by reflections on $\ell$, the latter with kernel $H$. Nevertheless, as we shall see below,  the proof of Dudley's result can be adapted to show that there is always an open subgroup fixing a vertex. 

Recall that the {\em infinite dihedral group} $D_\infty$ is the group
$$
D_\infty=\Z_2*\Z_2=\langle a,b\del a^2=b^2=1\rangle.
$$
We can also see $D_\infty$ as the group of all automorphisms of the biinfinite line $\ell=(\ldots,x_{-1},x_0,x_1,\ldots)$ generated by a reflection $a=R_0$ around $x_0$ and a reflection $b=R_{1/2}$ around the midpoint between $x_0$ and $x_1$. Note also that if we instead let $b$ be the reflection $R_1$ around $x_1$, we obtain an action of $D_\infty$ without inversion on $\ell$. We thus have a bijective correspondence between actions without inversion on $\ell$ and homomorphisms into $D_\infty$.

\begin{thm}\label{dihedral}
Let $G$ be a completely metrisable  group acting without inversion on the biinfinite line $\ell$. Then there is an open subgroup fixing a vertex.
\end{thm}

\begin{proof}
Suppose first that any open $V\ni 1$ contains a hyperbolic element. Let also $U\ni 1$ be open and $n\geqslant 1$. Pick then  $V\ni 1$ open such that $V^n\subseteq U$. So if $g\in V$ is hyperbolic, $g^n\in U$ is hyperbolic of amplitude $\|g\|\geqslant  n$.
Suppose now instead that for any open $V\ni 1$ there are elliptic $g,h\in V$ with no common fixed point. Then for such $g,h\in V$, $gh\in V^2$ is hyperbolic. It follows from these considerations that if no open subgroup of $G$ fixes a vertex, then any open neighbourhood of $1$ contains hyperbolic elements of arbitrarily large amplitude.

So suppose towards a contradiction that $G$ has no open subgroup fixing a vertex and find a sequence $g_i\in G$ converging to $1$ such that, on the other hand, $\|g_i\|\til \infty$. By passing to a subsequence, we can suppose that there are {\em even} numbers $k_i\geqslant 1$ such that the infinite products
$$
y_m=g_m(g_{m+1}(g_{m+2}(\ldots)^{k_{m+2}})^{k_{m+1}})^{k_m}\in G
$$
exist and the following are satisfied
\begin{enumerate}
  \item $k_m<\|g_{m+1}\|$,
  \item $y_m=g_my_{m+1}^{k_m}$,
  \item $k_m=m+\sum_{i=1}^m\|g_i\|$.
\end{enumerate}
Now, if $y_{m+1}$ elliptic, then $y_{m+1}^2$ acts trivially on $\ell$, whence, as $k_m$ is even,
$$
\|y_m\|=\|g_my_{m+1}^{k_m}\|=\|g_m\|>k_{m-1}.
$$
On the other hand,  if $y_{m+1}$ is hyperbolic, then, since $g_m$ is hyperbolic too,
$$
\|y_m\|=\|g_my_{m+1}^{k_m}\|\geqslant k_m\cdot\|y_{m+1}\|-\|g_m\|\geqslant k_m-\|g_m\|>k_{m-1}.
$$
It follows that for all $m$,
\begin{displaymath}\begin{split}
\|y_1\|&=\|g_1y_2^{k_1}\|\\
&\geqslant \|y_2^{k_1}\|-\|g_1\|\\
&\geqslant \|y_2\|-\|g_1\|\\
&=\ldots\\
&\geqslant \|y_{m+1}\|-\sum_{i=1}^m\|g_i\|\\
&>k_m-(k_m-m)\\
&=m,
\end{split}\end{displaymath}
contradicting that $\|y_1\|$ is finite.
\end{proof}


\subsection{Proof of Theorem \ref{equiv}}\label{main}
For the proof of Theorem \ref{equiv}, we shall need the following basic calculation (this is  (6.14) in  Alperin and Bass \cite{alperin-bass}).
\begin{lemme}\label{amplitude}
Suppose a group $G$ acts without inversion on a tree $X$ and let $x\in {\rm Vert}\:X$. Then the amplitude $\|\cdot\|\colon G\til \N$ is given by
$$
\|g\|=\max\big\{0, \;d(x, g^2(x))-d(x,g(x))\big\}.
$$
\end{lemme}

\begin{proof}
Fix $x\in {\rm Vert}\;X$. Suppose $g\in G$ is hyperbolic. Then as can be seen from Figure \ref{isometric}, we have 
$\|g\|=d(x,g^2(x))-d(x,g(x))$.

On the other hand, if $g$ is elliptic, let $y\in X_g$ be the vertex in $X_g$ closest to $x$. Then, the geodesic path from $x$ to $g(x)$ is the concatenation of $x-y$ and $y-g(x)$, while, on the other hand, the concatenation of $x-y$ and $y-g^2(x)$ also provides a path from $x$ to $g^2(x)$, though perhaps not the shortest (see Figure \ref{isometric}). Since $d(y,x)=d(y,g(x))=d(y,g^2(x))$, we see that
$$
d(x,g^2(x))\leqslant d(x,g(x)).
$$
In any case, 
$$
\|g\|=\max\big\{0, \;d(x, g^2(x))-d(x,g(x))\big\}.
$$

\setlength{\unitlength}{.5cm}
\begin{figure}
\begin{picture}(20,9.5)(0,-4)

\put(-2,-2){\vector(1,0){10}}
\put(0,-2){\line(0,1){1.5}}
\put(3,-2){\line(0,1){1.5}}
\put(6,-2){\line(0,1){1.5}}
\put(8.5,-2.2){$\ell_g$}
\put(-.16,-0.7){$\bullet$}
\put(2.86,-.7){$\bullet$}
\put(5.86,-.7){$\bullet$}
\put(-.2,-.2){${x}$}
\put(2.7,-.2){$ g(x)$}
\put(5.7,-.2){$ g^2(x)$}

\put(2,-4){$g$ \text{hyperbolic}}
\put(14.5,-4){$g$ \text{elliptic}}

\put(15,-2){$\bullet$}
\put(15.16,-1.85){\line(-1,2){1.7}}
\put(15.16,-1.85){\line(1,2){1.7}}

\put(15.2,1.6){\line(-1,-2){.87}}
\put(14.5,2){$g^2(x)$}
\put(15,1.35){$\bullet$}

\put(13.35,1.35){$\bullet$}

\put(16.7,1.35){$\bullet$}

\put(15,-2.5){$y\in X_g$}
\put(13.2,2){$x$}

\put(16.6,2){$g(x)$}
\end{picture}
\caption{}\label{isometric}
\end{figure}

\end{proof}

\begin{cor}\label{continuous amplitude}
Let $G$ be a topological group acting without inversion on a tree $X$. Suppose that for some $x\in {\rm Vert}\;X$ the map 
$$
g\in G\mapsto d(x,g(x))\in \N
$$
is continuous. Then the amplitude $\|\cdot\|\colon G\til \N$ is continuous.
\end{cor}

Let us first recall the statement of Theorem \ref{equiv}.

\begin{thm}
Let $G$ be a completely metrisable or locally compact group acting without inversion on  a tree $X$. Then the following four properties are equivalent.
\begin{enumerate}
\item $G$ either fixes an end or has an open subgroup fixing a vertex,
\item the amplitude $\|\cdot\|\colon G\til \N$ is a continuous function on $G$,
\item the set of elliptic isometries is somewhere comeagre in $G$,
\item there is an open subgroup consisting of elliptic isometries.
\end{enumerate}
Moreover, the above properties hold for any action of $G$ if and only if whenever $G$
is written as a non-trivial free product with amalgamation, $G=A*_CB$,  the subgroups $A$, $B$, and $C$ are open in $G$.
\end{thm}

\begin{proof}Suppose $G$ acts without inversion on a tree $X$.

(1)$\saa$(2): Assume $G$ satisfies (1). We show that $\|\cdot\|\colon G\til \N$ is continuous.

Suppose first that $\bf e$ is an end fixed by $G$ and define a homomorphism $\pi\colon G\til \Z$ as follows
$$
\pi(g)=\begin{cases}
0; &\text{if $g$ is elliptic,}\\
\|g\|; &\text{if $g$ is hyperbolic and $\bf e$ is an attracting fixed end for $g$,}\\
-\|g\|; &\text{if $g$ is hyperbolic and $\bf e$ is a repulsing fixed end for $g$.}
\end{cases}
$$
Since $\bf e$ is a fixed end for $g$, this is easily seen to be a homomorphism, and so, by  Theorem \ref{dudley}, $\pi$ is continuous. Thus, also $\|\cdot\|\colon G\til \N$ given by $\|g\|=|\pi(g)|$ is continuous. 

Now suppose instead that $K$ is an open subgroup of $G$ fixing a vertex $x\in {\rm Vert}\:X$. Note that if $f,g\in G$ satisfy $g\inv f, g^{-2}f^2\in K$, then $g(x)=f(x)$ and $g^2(x)=f^2(x)$, whence, by Lemma \ref{amplitude}, $\|g\|=\|f\|$. Since $K$ is open it follows that $\|\cdot\|$ is locally constant and thus continuous.

(2)$\saa$(3):  Assume that $\|\cdot\|$ is continuous. Then the set $E=\{g\in G\del \|g\|=0\}$ of elliptic isometries is non-empty open and hence somewhere comeagre in $G$.

(3)$\saa$(4): We split the proof of this implication into several lemmas.
\begin{lemme}\label{finite segment}
Suppose $\go s=(x_0,x_1,\ldots,x_n)$ is a finite line segment of $X$ and $B\subseteq G$ is a somewhere comeagre set of elliptic isometries such that $X_g\cap \go s\neq \tom$ for all $g\in B$. Then there is an open subgroup $H\leqslant G$ that fixes a single vertex $x\in {\rm Vert}\;X$.
\end{lemme}

\begin{proof}
Let $\{x_i,x_{i+1},\ldots,x_j\}\subseteq \{x_0,x_1,\ldots,x_{n}\}$ be a segment, minimal with respect to the property that for a somewhere comeagre subset of elliptic isometries $D\subseteq G$ we have
$$
X_g\cap \{x_i,x_{i+1},\ldots,x_j\}\neq \tom
$$
for all $g\in D$. Fix also the corresponding set $D$ and assume toward a contradiction that $i<j$.

Now, choose a non-empty open subset $U\subseteq G$ in which $D$ is comeagre  and fix a non-empty open subset $V\subseteq U$ such that $V^{2j+1}V^{-2j}\subseteq U$. By the minimality of $ \{x_i,x_{i+1},\ldots,x_j\}$, we can find $g,h\in D\cap V$ such that
$$
x_i\in X_g, \quad x_{i+1}\notin X_g, \quad x_{j-1}\notin X_h, \quad x_j\in X_h.
$$
Since then $X_g\cap X_h=\tom$,  $hg$ is a hyperbolic isometry of amplitude
$$
\|hg\|=2\cdot{\rm dist}(X_h,X_g)=2\cdot d(x_j,x_i)\geqslant 2
$$
along a line $\ell_{hg}$ containing the geodesic path from $X_g$ to $X_h$, i.e., $ \{x_i,x_{i+1},\ldots,x_j\}$. By inspection one sees that $hg$ translates the line $\ell_{hg}$ in the direction from $x_i$ toward $x_j$ (assuming that the action of $G$ on $X$ is a left-action). It follows that if $Y$ is a subtree of $X$ containing some $x_p$, $i<p\leqslant j$ but not containing $x_i$, then
$$
(hg)^j(Y)\cap  \{x_i,x_{i+1},\ldots,x_j\}=\tom.
$$
Note now that as $V^{2j+1}V^{-2j}\subseteq U$, $g,h\in V$ and  $D$ is comeagre in $U$, the set
$$
E=D\cap (hg)^{-j}D(hg)^{j}
$$
is comeagre in $V$. So, by the minimality of $ \{x_i,x_{i+1},\ldots,x_j\}$, we can find some $k\in E$ such that $x_i\notin X_k\cap \{x_i,x_{i+1},\ldots,x_j\}\neq \tom$. We then have that
$$
X_{(hg)^{j}k(hg)^{-j}}=(hg)^{j}(X_k)
$$
is disjoint from $\{x_i,\ldots,x_j\}$, contradicting that $(hg)^{j}k(hg)^{-j}\in D$.

Thus, our assumption that $i<j$ is incorrect and hence, for all $g\in D$, $x_i=x_j\in X_g$. Letting $H$ be the subgroup of $G$ generated by $D$, $H$ fixes $x_i$ and by Pettis' Theorem (see Theorem (9.9) in \cite{kechris}) is an open subgroup of $G$.
\end{proof}

\begin{lemme}\label{half line}
Suppose $\ell^+=(x_0,x_1,x_2,\ldots)$ is an infinite half-line in $X$ and $B\subseteq G$ is a somewhere comeagre set of elliptic isometries such that $X_g\cap \ell^+\neq \tom$ for all $g\in B$. Then there is an open subgroup $H\leqslant G$ either fixing a vertex or having a neutral fixed end.
\end{lemme}

\begin{proof}
Fix a non-empty open set $U$ in which $B$ is comeagre and let $V\subseteq U$ be a non-empty open subset such that $VVV\inv \subseteq U$.

If for all $f\in B\cap V$, $X_f$ contains a tail of $\ell^+$, then if $H$ is the subgroup of $G$ generated by $B\cap V$, $H$ is open by Pettis' Theorem and $\ell^+$ is a neutral fixed end for $H$.

So suppose  instead that there is some  $f\in B\cap V$ such that no tail of $\ell^+$ is contained in $X_f$ and let $n$ be maximal such that $x_n\in X_f$.

\begin{claim}
For a somewhere comeagre subset $C\subseteq B$ and all $g\in C$,
$$
X_g\cap \{x_{0}, x_{1},\ldots, x_{2n}\}\neq \tom.
$$
\end{claim}

\begin{proof}
Note that by assumption on $n$,
$$
f(\{x_{n+1}, x_{n+2},x_{n+3},\ldots\})\cap  \{x_{n+1}, x_{n+2},x_{n+3},\ldots\}=\tom.
$$
Now, let $Y\subseteq X$ be the subtree of $X$ consisting of all $y$ whose geodesic path to $x_n$ contains the line segment $(x_{2n+1},x_{2n},\ldots,x_n)$. Since for any $y\in Y$,
$$
d(f(y),x_n)=d(f(y),f(x_n))=d(y,x_n)\geqslant d(x_{2n+1},x_n)>n,
$$
we see that $f(Y)\cap \{x_0,\ldots,x_n\}=\tom$. Moreover, for any $y\in Y$, the geodesic path from $f(y)$ to $f(x_n)=x_n$ contains the line segment $f(\{x_{2n+1},x_{2n},\ldots,x_n\})$, so as
$$
f(\{x_{n+1}, x_{n+2},x_{n+3},\ldots\})\cap  \{x_{n+1}, x_{n+2},x_{n+3},\ldots\}=\tom,
$$
we have $f(Y)\cap \ell^+=\tom$.

Now, since $f\in V$, we have $fVf\inv \subseteq U$, so $C=B\cap f\inv Bf$   is comeagre in $V$. It follows that for every $g\in C$, $X_g\cap \{x_{0}, x_{1},\ldots, x_{2n}\}\neq \tom$.
For assume this fails for some $g\in C$. Then, since $X_g\cap \ell^+\neq\tom$, 
$$
X_g\cap \{x_{2n+1},x_{2n+2},x_{2n+3},\ldots\}\neq \tom
$$
and hence $X_g\subseteq Y$. Now $fgf\inv \in B$, but
$$
X_{fgf\inv}\cap \ell^+=f(X_g)\cap \ell^+\subseteq f(Y)\cap \ell^+=\tom,
$$
which contradicts the assumption that $X_h\cap \ell^+\neq \tom$ for all $h\in B$.
\end{proof}

Now applying Lemma \ref{finite segment} to the somewhere comeagre set $C$ and the finite line segment $\go s=(x_0,\ldots,x_{2n})$, we obtain an open subgroup $H$ of $G$ fixing a vertex of $X$.
\end{proof}

\begin{lemme}\label{line}
Suppose $\ell=(\ldots,x_{-2},x_{-1},x_0,x_1,x_2,\ldots)$ is a biinfinite line in $X$ and $B\subseteq G$ is a somewhere comeagre set of elliptic isometries such that $X_g\cap \ell\neq \tom$ for all $g\in B$. Then there is an open subgroup $H\leqslant G$ either fixing a vertex or having a neutral fixed end.
\end{lemme}

\begin{proof}
Let $U\subseteq G$ be a non-empty open set in which $B$ is comeagre and let $V\subseteq U$ be a non-empty open subset such that $VVV\inv\subseteq U$.

Assume first that for some $f\in B\cap V$, $f(\ell)\neq \ell$ and pick some $x_n\in \ell$ such that $f(x_n)\notin \ell$. Then either
$$
f(\{x_n,x_{n+1},x_{n+2},\ldots\})\cap \ell=\tom
$$
or
$$
f(\{\ldots,x_{n-2},x_{n-1},x_n\})\cap \ell=\tom.
$$
Without loss of generality, we can suppose the first option holds. Notice that, since $B$ is comeagre in $U$ and $fVf\inv\subseteq U$, $B\cap f\inv Bf$ is comeagre in $V$. We claim that if $g\in B\cap f\inv Bf$, then
$$
X_g\cap \{\ldots,x_{n-3},x_{n-2},x_{n-1}\}\neq \tom.
$$
If not, pick a $g$ for which it fails. Since $g\in B$, $X_g\cap \ell\neq \tom$ and hence
$$
X_g\cap \{x_{n},x_{n+1},x_{n+2},\ldots\}\neq \tom.
$$
So $X_g$ is contained in the subtree of $X$ consisting of all $y\in X$ whose geodesic path to $\{\ldots,x_{n-3},x_{n-2},x_{n-1}\}$ passes through $x_n$. By choice of $f$ it follows that $f(X_g)\cap \ell=\tom$. But $f(X_g)=X_{fgf\inv}$ and $fgf\inv\in B$, contradicting the assumption on $B$.

It thus follows that there is a half-line $\ell^+= (x_{n},x_{n+1},x_{n+2},\ldots)$ in $X$ such that for a somewhere comeagre set $C=B\cap f\inv Bf$ of elliptic isometries we have $X_g\cap \ell^+\neq \tom$ for all $g\in C$. By Lemma \ref{half line}, the conclusion of the lemma follows.

Now assume instead that for all $f\in B\cap V$, $f(\ell)=\ell$. Then, if $K\leqslant G$ is the subgroup of $G$ generated by $B\cap V$, $K$ will be open by Pettis' Theorem and, moreover, $\ell$ is invariant under $K$. Now, if $G$ is completely metrisable, so is $K$, and hence, by Theorem \ref{dihedral}, there is an open subgroup $H\leqslant K$ fixing a vertex of $\ell\subseteq X$. If on the other hand, $G$ and thus also $K$ are locally compact, as we shall see in Theorem \ref{locally compact}, $K$ has an open subgroup $H$ with property (FA') that therefore also fixes a vertex of $\ell\subseteq X$.
\end{proof}

Now suppose finally that the set $E\subseteq G$ of elliptic isometries is somewhere comeagre. We will show that $G$ has an open subgroup $H$ that either fixes a vertex or has a neutral fixed end. So, let $U\subseteq G$ be a non-empty open set in which $E$ is comeagre and let $V\subseteq U$ be a non-empty open subset such that $VVV\inv\subseteq U$.

Suppose first that the trees $\{X_f\}_{f\in E\cap V}$ intersect pairwise. Then, if there is some $x\in \bigcap_{f\in E\cap V}X_f$,  the somewhere comeagre set $E\cap V$ and hence also the open subgroup $H\leqslant G$ generated by $E\cap V$ fixes $x$. Otherwise, if $\bigcap_{f\in E\cap V}X_f=\tom$, there is an infinite half-line $\ell^+$ such that any $X_g$, $g\in E\cap V$, contains a tail of $\ell^+$, and so the open subgroup $H\leqslant G$ generated by $E\cap V$ has a neutral fixed end.

So suppose instead that for some $f,g\in E\cap V$,  $X_{f}\cap X_{g}= \tom$. Then $fg\inv$ is a hyperbolic isometry acting by translation on some line $\ell_{fg\inv}\subseteq X$. Also, for all $h\in E\cap Egf\inv$, $X_h\cap \ell_{fg\inv}\neq \tom$, for otherwise, $hfg\inv$ will be hyperbolic, contradicting that $hfg\inv \in E$. Applying Lemma \ref{line} to the somewhere comeagre set $E\cap Egf\inv$, we again find an open subgroup $H$ either fixing a vertex or having a neutral fixed end.

(4)$\saa$(1): This follows from Proposition 5 in \cite{alperin2}, but for the convenience of the reader, we shall reproduce the argument here.
Suppose that $H\leqslant G$ is an open subgroup consisting of elliptic isometries and that $G$ has no open subgroup fixing a vertex of $X$. Then, in particular, $H$ fixes no vertex and hence must have a neutral fixed end $\bf e$. We claim that $\bf e$ is fixed by $G$. To see this, suppose toward a contradiction that for some $g\in G$, we have $g\cdot \bf e\neq \bf e$. Then $g\cdot \bf e$ is a neutral fixed end for $gHg\inv$, whereby both $\bf e$ and $g\cdot \bf e$ are neutral fixed ends for the open subgroup $N=H\cap gHg\inv\leqslant G$. But, as $g\cdot\bf e\neq e$, there is a unique biinfinite line $\ell$ sharing  tails with both $\bf e$ and with $g\cdot \bf e$. It follows that any $h\in N$ fixes two opposite tails of $\ell$, whereby $h$ fixes every vertex of $\ell$. This contradicts that no open subgroup of $G$ fixes a vertex.

\

For the moreover part, suppose that whenever $G$ acts without inversion on a tree properties (1), (2), (3), and (4) hold. Assume $G=A*_CB$ is a decomposition of $G$ as a non-trivial free product with amalgamation and let $X$ be the Bass-Serre tree (see Theorem 7, p. 32 \cite{serre}) corresponding to this decomposition having vertices 
$$
{\rm Vert}\; X=G/A\sqcup G/B
$$ 
and with edges between $gA$ and $gB$ for any $g\in G$.   By assumption on $G$, there is an open subgroup $H\leqslant G$ consisting of elliptic isometries. Also, since $C$ is a subgroup of  $A$ and $B$,  it suffices to show that  $C$ is open in $G$. 

Suppose first that $H$ fixes a vertex of $X$. Then, by construction of $X$, $H$ is contained in a conjugate of either $A$ or of $B$, say $H\leqslant gAg\inv$. Pick $b\in B\setminus C$. Then 
$$
g\inv Hg\cap bg\inv Hgb\inv\leqslant A\cap bAb\inv\leqslant C
$$
and so $C$ is open too.

Suppose now instead that $H\leqslant G$ fixes an end $\bf e$. We can write ${\bf e}=[\ell^+]$, where
$$
\ell^+=(1A, a_0B, a_0b_0A,a_0b_0a_1B,a_0b_0a_1b_1A,\ldots)
$$
for $a_i\in A$ and $b_i\in B$. Since $B$ is a proper subgroup of $G$, it has index at least $2$ and hence there is some $a\in A$ such that $aB\neq a_0B$. It follows that
$$
aa_0\inv\cdot \ell^+=(1A, aB, ab_0A,ab_0a_1B,ab_0a_1b_1A,\ldots)
$$
intersects $\ell^+$ only in the vertex $1A\in {\rm Vert}\;X$. So $N=H\cap aa_0\inv Ha_0a\inv$ is an open subgroup of elliptic isometries fixing the biinfinite  line 
$$
\ell=\ell^+\cup aa_0\inv \cdot \ell^+.
$$
Thus $N$ must be contained in a conjugate of $C$, whereby $C$ is open too.

For the converse implication,  suppose that whenever $G=A*_CB$ is a decomposition as a non-trivial free product with amalgamation, the subgroups $A$, $B$, and $C$ are open in $G$.  Assume also that $G$ acts without inversion on a tree $X$.  By Theorem \ref{serre2}, if $R$ denotes the subgroup of $G$ generated by elliptic isometries, then $G/R\iso \pi_1(G\setminus X)$. So, as $\pi_1(G\setminus X)$  is a free group, by Theorem \ref{dudley} the quotient map from $G$ onto $G/R$ has open kernel, whence $R$ is an open subgroup of $G$. Now, applying Theorem \ref{serre2}, to the action of $R$ on $X$, we see that $R\setminus X$ is a tree.  By Proposition 17, p. 32 of \cite{serre}, we can find a fundamental domain $T\subseteq X$ for the action of $R$ on $X$, whence $T$ is a subtree of $X$. We let for every $x\in {\rm Vert}\:T$, $R_x={\rm Stab}_R(x)$, and for every $e\in {\rm Edge}\:T$, 
$$
R_e=R_{\overline e}={\rm Stab}_R(e)=R_{o(e)}\cap R_{t(e)}.
$$ 
So the inclusion homomorphisms $\alpha_e\colon R_e\inj R_{o(e)}$, $\alpha_{\overline e} \colon R_{\overline e}\inj R_{t(e)}$ are monomorphisms. This defines a tree of groups $(R,T)$, and by Theorem 10, p. 39 of \cite{serre}, we have
$$
R\iso \lim_{\longrightarrow}(R,T).
$$
In other words, $R$ is the free product of the groups $\{R_x\}_{x\in {\rm Vert}\;T}$ subject to the additional relations
$$
\alpha_{\overline e}(g)=\alpha_{e}(g)
$$
for $e\in {\rm Edge}\;T$ and $g\in R_e=R_{\overline e}$.

Now, for any edge $e\in {\rm Edge}\;T$, let $T_e$ be the subtree of $T$ spanned by the set of vertices $x$ whose geodesic to $t(e)$ passes through $e$. So, for any edge $e\in {\rm Edge}\;T$, the set of vertices of $T$ decomposes as
$$
{\rm Vert}\; T={\rm Vert}\;T_e\sqcup {\rm Vert}\:T_{\overline e},
$$
whence, if $A_e=\langle R_x\del x\in {\rm Vert}\;T_e\rangle$, then $R=A_e*_{R_e}A_{\overline e}$.

Now suppose that $G$ has no open subgroup fixing a vertex and hence no open subgroup fixing an edge. Then, by the basic assumption on $G$, we see that for any edge $e\in {\rm Edge}\;T$, the decomposition $R=A_e*_{R_e}A_{\overline e}$ is trivial, whence either 
\begin{equation}\label{orient1}
A_e=R_{o(t)}=R_e\leqslant R_{t(e)}\leqslant A_{\overline e}
\end{equation}
or
\begin{equation}\label{orient2}
A_{\overline e}=R_{t(e)}=R_e\leqslant R_{o(e)}\leqslant A_e.
\end{equation}
We note that if both $A_e=R_e$ and $R_e=A_{\overline e}$, then $R=R_e$, contradicting that no open subgroup of $G$ fixes an edge. Thus, for any edge $e\in {\rm Edge}\;T$, exactly one of \eqref{orient1} and \eqref{orient2} holds. This gives us a natural orientation of the edges of the tree $T$, namely by letting 
$$
{\rm Edge}^+T=\{e\in {\rm Edge}\;T\del A_e=R_{o(t)}=R_e\leqslant R_{t(e)}\leqslant A_{\overline e}\}.
$$

Assume towards a contradiction that there is a vertex $x\in {\rm Vert}\;T$ such that any edge whose terminal vertex is $x$ belongs to ${\rm Edge}^+T$. Then, as $R$ is generated by
$$
R_x\cup \bigcup\{A_e\del e\in {\rm Edge}\:T\;\&\; t(e)=x\}=R_x\cup \bigcup\{A_e\del e\in {\rm Edge}^+T\;\&\; t(e)=x\}=R_x,
$$
we see that $R=R_x$, contradicting that $G$ has no open subgroup fixing a vertex. 

Thus, for every vertex $x$, there is a positively oriented edge whose origin is $x$. We claim that moreover there is exactly one such edge for each vertex $x$. To see this, suppose $e_1,e_2\in {\rm Edge}^+T$  are distinct edges with $o(e_1)=o(e_2)=x$. Then, as $e_1$ is positively oriented,
$$
A_{\overline e_2}\leqslant A_{e_1} \leqslant R_{t(e_1)},
$$
whence  $A_{\overline e_2}=R_{t(e_2)}=R_{e_2}\leqslant R_x$, contradicting that $e_2$ is positively oriented.

So every vertex has exactly one outgoing, positively oriented vertex. This uniquely defines an end of the tree by following the positively oriented path originating from any vertex. I.e., let $x_0$ be an arbitrary vertex of $T$ and define inductively $x_n$ by letting $x_{n+1}$ be the terminal vertex of the unique positively oriented edge originating at $x_n$. Then, we see that if $(y_n)$ was a similarly defined path originating at another vertex $y_0$, the two paths $\ell^+=(x_n)$ and $(y_n)$ would have a common tail. Moreover, as
$$
R_{x_0}\leqslant R_{x_1}\leqslant R_{x_2}\leqslant\ldots
$$
and 
$$
R_{y_0}\leqslant R_{y_1}\leqslant R_{y_2}\leqslant\ldots,
$$
we see that $R_{y_0}\leqslant R_{x_n}$ for all but finitely many $n$. Since $R$ is generated by the vertex stabilisers, it follows that $R=\bigcup_{n\in \N}R_{x_n}$, whence ${\bf e}=[\ell^+]$ is a neutral fixed end for $R$. Thus, $R$ is an open subgroup of $G$ consisting entirely of elliptic isometries. In any circumstance, Property (4) holds.
\end{proof}


\section{Further results on completely metrisable and locally compact groups}\label{non-discrete}
There are a number of results in the literature treating various aspects of completely metrisable and locally compact groups acting on trees, though these are somewhat scattered and the authors seem to have worked mostly independently of each other.

Using Dudley's Theorem, S. A. Morris and P. Nickolas \cite{morris} showed that if $G$ is locally compact and $\pi\colon G\til A*B$ is a homomorphism into an arbitrary free product of groups, then either $\pi$ is continuous with respect to the discrete topology on $A*B$ or the image $\pi(G)$ lies within a conjugate of either $A$ or $B$.

In a completely unrelated work, H. Bass \cite{bass} was developing the structure theory for groups acting on trees and in this connection showed that profinite groups have property (FA'). On the other hand, S. Koppelberg and J. Tits \cite{tits} proved that if $F$ is a finite perfect group, the infinite product $F^\N$ is not the union of a countable chain of proper subgroups. Together, these results imply that $F^\N$ actually has property (FA), i.e., whenever $F^\N$ acts without inversion on a tree it fixes a vertex.

R. Alperin \cite{alperin1, alperin2} continued the study of properties (FA') and (FA) in the context of topological groups and showed that both compact groups and connected locally compact groups have property (FA').  

Recently there has been renewed interest in such questions from model theory. In particular, D. Macpherson and S. Thomas \cite{mactho} showed that is $G$ is a completely metrisable  group with a comeagre conjugacy class, then $G$ has property (FA') (a simple proof of this is given in \cite{OB}). A number of other authors, e.g., \cite{ivanov1, kossak,khelif,shelah} give various results related to both Dudley's Theorem and the result of Macpherson and Thomas. 

On a somewhat different note, G. M. Bergman \cite{bergman} showed that whenever the infinite symmetric group $S_\infty$ acts by isometries on a metric space without any assumption of continuity, all orbits are bounded. It easily follows from this that $S_\infty$ has property (FA) (which was proved earlier by J. Saxl, S. Shelah and S. Thomas in \cite{saxl}) and has even stronger fixed point properties. Bergman's result has now been verified for a number of other non-locally compact groups by various authors.

Let us first settle the situation for locally compact groups by deducing from Alperin's results.
\begin{thm}\label{locally compact}
Any locally compact, Hausdorff topological group has an open subgroup with property (FA').
\end{thm}

\begin{proof}
Let $G_0$ be the connected component of the identity in $G$. Since $G_0$ is connected and locally compact, Hausdorff, $G_0$ has property (FA') by the results of Alperin \cite{alperin2}. Moreover, $G/G_0$ is a locally compact, totally disconnected group, so, by  a theorem of van Dantzig (\cite{hofmann}, Theorem 1.34), $G/G_0$ has a neighbourhood basis at the identity consisting of compact open subgroups. So let $K\leqslant G/G_0$ be compact open and let $H\leqslant G$ be the preimage of $K$ under the projection map from $G$ to $G/G_0$. Then $H$ is an open subgroup of $G$ containing $G_0$. By the main result of Alperin \cite{alperin1}, $K=H/G_0$ has property (FA'). Since both $H/G_0$ and $G_0$ have property (FA'), it follows from Corollary 2 of \cite{alperin2} that also $H$ has (FA').
\end{proof}
So, in particular, any action without inversion of a locally compact group on a tree satisfies the equivalent conditions of Theorem \ref{equiv}.  For completely metrisable groups, we have not been able to decide whether this is the case, though we can prove it under fairly mild additional assumptions on the group.

In the following, $G$ will be a completely metrisable group acting without inversion on a tree $X$.
We shall begin by an easy observation (see H. Bass \cite{bass}).
\begin{lemme}
Suppose $g\in G$ is hyperbolic and $H\leqslant G$ is a subgroup containing $g$ such that $H\cdot \ell_g=\ell_g$. Then $N_G(H)\cdot \ell_g=\ell_g$.
\end{lemme}

\begin{proof}
Suppose $f\in N_G(H)$, i.e., that $fHf\inv=H$. Then $fgf\inv\in H$ is hyperbolic, and so the unique $fgf\inv$-invariant line must equal $\ell_g$.
In other words, $f\cdot \ell_g=\ell_{fgf\inv}=\ell_g$.
\end{proof}

\begin{lemme}\label{bass}
Every hyperbolic isometry $g\in G$ generates an infinite discrete subgroup of $G$.
\end{lemme}

\begin{proof}Note that, as $G$ is Hausdorff,  $\overline{\langle g\rangle}$ is contained in $N_G(\langle g\rangle)$ and that $\langle g\rangle \cdot \ell_g=\ell_g$. Therefore, by the preceding lemma, we also have $\overline{\langle g\rangle}\cdot \ell_g=\ell_g$. This therefore defines an action by $\overline{\langle g\rangle}$ without inversion on the line $\ell_g$. By Theorem \ref{dihedral}, there is an open subgroup $K\leqslant \overline{\langle g\rangle}$ fixing a vertex, whence $K\cap \langle g\rangle =\{1\}$, showing that $1$ is isolated in $\langle g\rangle$. It follows that $g$ generates an infinite discrete subgroup of $G$.
\end{proof}

We are now ready for the proof of Theorem \ref{somewhere comeagre}.
\begin{thm}
Let $G$ be a completely metrisable group such that the set 
$$
\D=\{g\in G\del \langle g\rangle \textrm{ is either finite or non-discrete }\}
$$
is somewhere dense in $G$. Then $G$ has property (OA).
\end{thm}

\begin{proof}Note that by Lemma \ref{bass}, whenever $G$ acts without inversion on a tree, every element of $\D$ is elliptic. So, by Theorem \ref{equiv}, it suffices to show that $\D$ is somewhere comeagre. But, if  $\{V_k\}_{k\in \N}$ is  a neighbourhood basis at the identity of $G$, then
\begin{align*}
\D=&\{g\in G\del \langle g\rangle \textrm{ is either finite or non-discrete }\}\\
=&\bigcap_{k\in \N}\{g\in G\del \e n\geqslant 1\; g^n\in V_k\},
\end{align*}
is $G_\delta$ and somewhere dense, whence somewhere comeagre.
\end{proof}
Let us note that Conjecture \ref{conj}, stating that all completely metrisable groups have property (OA), would imply that any completely metrisable group $G$, that is either connected or has a dense conjugacy class, will have property (FA'). For in both cases, whenever $G$ acts without inversion on a tree, the amplitude would be constant and hence constantly $0$.

Theorem \ref{somewhere comeagre} can be strengthened in case the set $\D$ is actually dense.
\begin{thm}\label{dense}
Let $G$ be a completely metrisable group such that the set 
$$
\D=\{g\in G\del \langle g\rangle \textrm{ is either finite or non-discrete }\}
$$
is dense in $G$. Then $G$ has property (FA').
\end{thm}

\begin{proof}Suppose $G$ acts without inversion on a tree $X$. Then, by Theorem \ref{somewhere comeagre}, the amplitude function $\|\cdot\|\colon G\til \N$ is continuous on $G$. Moreover,  by Lemma \ref{bass}, $\|\cdot\|$ is constantly $0$ on the dense set $\D$ and hence constantly $0$ on $G$. So every element of $G$ is elliptic.
\end{proof}

The class of completely metrisable groups that satisfy the hypotheses of Theorems \ref{somewhere comeagre} and \ref{dense} is fairly rich and somehow orthogonal to the class of Polish groups with a comeagre conjugacy class. Of course, any element of a compact metric group generates a finite or non-discrete subgroup, but also in very large groups, such as the group of measure-preserving automorphisms of the unit interval with Lebesgue measure equipped with the weak topology, ${\rm Aut}([0,1],\lambda)$, the unitary group, $U(\ell_2)$, and the isometry group of the Urysohn metric space, ${\rm Isom}(\U)$, the generic element generates  non-discrete subgroup (see \cite{powers} for a more complete discussion of this phenomenon).


\end{document}